\documentclass[12pt]{amsart}
\usepackage{graphicx}
\usepackage{amssymb,amsmath,enumerate,amsthm,soul}
\usepackage{epstopdf,cancel,algorithm,algpseudocode,color}

\newtheorem{thm}{Theorem}[section]
\newtheorem{cor}[thm]{Corollary}
\newtheorem{lem}[thm]{Lemma}
\newtheorem{prop}[thm]{Proposition}
\newtheorem{defn}[thm]{Definition}
\theoremstyle{remark}
\newtheorem{rem}[thm]{Remark}

\begin{document}

\title{An Enumeration Process for Racks}

\author{Jim Hoste}
\address{Pitzer College,
1050 N Mills Avenue,
Claremont, CA 91711}
\email{jhoste@pitzer.edu}

\author{Patrick D. Shanahan}
\address{Loyola Marymount University, Dept.~of Mathematics UHall 2700, Los Angeles, CA 90045}
\email{pshanahan@lmu.edu}

\maketitle
\centerline{\today
}

\begin{abstract} Given a presentation for a rack $\mathcal R$, we define a process which systematically enumerates the elements of $\mathcal R$. The process is modeled on the systematic enumeration of cosets first given by Todd and Coxeter. This generalizes and improves the diagramming method for $n$-quandles introduced by Winker.  We provide pseudocode that is similar to that given by Holt for the Todd-Coxeter process.  We prove that the process terminates if and only if $\mathcal R$ is finite,  in which case, the procedure outputs an operation table for the finite rack. We conclude with an application to knot theory.
\end{abstract}

\section{Introduction}

The fundamental quandle of an oriented knot or link is an algebraic invariant which was proven to be a complete invariant of knots (up to mirror reversal) by Joyce \cite{J2}. See also Matveev \cite{Ma}. While it is easy to find a presentation of the quandle of a link using a modification of the Wirtinger algorithm, it is usually  difficult to determine the quandle's isomorphism class. A more tractable, but less sensitive, invariant is the $n$-quandle of a link which is a certain quotient of the fundamental quandle.  

In his Ph.\,D. thesis \cite{WI}, Winker introduced a method to produce a Cayley diagram of the $n$-quandle of a link. His diagramming method is a graph-theoretic modification of a fundamental process in computational group theory called the Todd-Coxeter process \cite{TC}.  This process was introduced to find the index of a finitely generated subgroup $H$ in a finitely presented group $G$. In addition, the process produces a table which describes the right action of $G$ on the set of cosets of $H$. The process is incorporated in many computer algebra systems.

Sarah Yoseph made a preliminary investigation of a Todd-Coxeter like process for the enumeration of $n$-quandles in her (unpublished) undergraduate senior thesis directed by the second author. Her work complemented Winker's by considering a table-based approach to $n$-quandle enumeration and producing elementary pseudocode. In this paper, we apply the table-based approach to the more general structure of a rack.  Our development of an enumeration process for a rack $\mathcal{R}$ given by a presentation $\langle S \, | \, R \, \rangle$ will be modeled on the exposition of the Todd-Coxeter process given in Holt \cite{HO}.  The rack enumeration process we present extends Winker's work to the study of racks and provides pseudocode for its implementation.

An important feature of the currently accepted Todd-Coxeter process is that if the index of $H$ is finite, then the process will terminate in a finite number of steps. In \cite{WA}, Ward showed that this was not true of the original process and provided a modification to the process to eliminate this problem. Using arguments similar to those in \cite{HO}, we prove that if our rack enumeration process completes, then the resulting output is rack isomorphic to $\mathcal R$ and, moreover, that the process completes if and only if $\mathcal R$ is finite. We also provide an example demonstrating the importance of Ward's modification in the rack setting as well.

In the special case of quandles,  the Todd-Coxeter process could be used in theory to determine the structure of any finite quandle. This is because Joyce proved that every quandle $Q$ is isomorphic to a quandle structure on the set of cosets of a particular subgroup of the automorphism group of $Q$. However, employing this approach would require determining a presentation for $\mbox{Aut}(Q)$ and generators for the appropriate subgroups which may not be practical. In the case of knot and link quandles, Joyce also proved that the coset quandle of the peripheral subgroups of the fundamental group is isomorphic to the fundamental quandle of the link. The authors extend this result to $n$-quandles of links in \cite{HS}. Hence, the Todd-Coxeter process can be used to investigate the structure of the $n$-quandle of a link, giving an alternative to Winker's method.  Given these theoretical and practical limitations, it is desirable to have an enumeration  procedure which applies directly to any finitely presented rack.

In Section~2 we review the basic definitions of racks and rack presentations. We introduce enumeration tables and the rack enumeration process in Section~3. We prove that the tables produced satisfy five basic properties which are used later to prove the main result in Section~4.  We also include pseudocode for the processes introduced in this section. Finally, in Section~4, we prove that if the process completes, then the output is isomorphic to the rack and, moreover,  that  a finitely presented rack is finite if and only if the process completes. In Section~5, we provide an example showing the importance of Ward's modification in the rack setting and discuss an alternative modification. We conclude with an application to knot theory. The authors thank the referees for their detailed and helpful comments on the article.

\section{Racks and presentations}

We begin with the definition and some basic properties of racks. Excellent sources for this material are \cite{FR}, \cite{EN}, \cite{J1}, \cite{J2}, and \cite{WI}.

\begin{defn}
\label{defnrack}
A set $\mathcal R$ with two binary operations $\rhd$ and $\rhd^{-1}$ is a {\bf rack} if the following two properties hold:
\begin{enumerate}
\item[\bf R1.] $(x \rhd y) \rhd^{-1} y = (x\rhd^{-1}y)\rhd y = x$ for all $x, y \in \mathcal R$, and
\item[\bf R2.] $(x\rhd y) \rhd z = (x\rhd z) \rhd (y \rhd z)$ for all $x, y, z \in \mathcal R$. 
\end{enumerate}
\end{defn}

\noindent Properties R1 and R2 are sometimes referred to as the right cancellation and right self-distributive axioms, respectively. It is easy to show that R1 and R2 imply $(x\rhd^\epsilon y) \rhd^\delta z = (x\rhd^\delta z) \rhd^\epsilon (y \rhd^\delta z)$ for $\epsilon, \delta \in \{-1,1 \}$. In general, a rack is non-associative and the following well-known lemma can be used to rewrite any product as a left-associated product.

\begin{lem} \label{qprop}
Let $\mathcal R$ be a rack and $x, y, z \in \mathcal{R}$. If $\epsilon, \delta \in \{-1,1 \}$, then 
$x \rhd^{\epsilon} (y \rhd^{\delta} z) = ((x \rhd ^ {-\delta} z) \rhd^{\epsilon} y) \rhd^{\delta} z.$
\end{lem}

\begin{proof} Using the cancellation and one of the distributive properties we have:
$$x \rhd^{\epsilon} (y \rhd^{\delta} z) = ((x \rhd^{-\delta} z) \rhd^{\delta} z)\rhd^{\epsilon} (y \rhd^{\delta} z)= ((x \rhd ^ {-\delta} z) \rhd^{\epsilon} y) \rhd^{\delta} z.$$
\end{proof}

A convenient notation introduced by Fenn 
and Rourke in \cite{FR} uses Lemma~\ref{qprop} to avoid the use of parentheses.  From this point on, we shall adopt Fenn and Rourke's {\bf exponential notation} defined by
$$x^y = x \rhd y  \ \ \ \mbox{and}\ \ \ x^ {\bar y} = x \rhd ^{-1} y.$$
With this notation, $x^{yz}$ will represent $(x^y)^z = (x \rhd y) \rhd z$, whereas, by Lemma~\ref{qprop}, $x^{\bar z y z}$ will be used to represent $x^{(y^z)}= x \rhd ( y \rhd z)$.

Given an integer $m$, we will also let $x^{y^m}$ denote $x^{y \dots y}$ if $m>0$, $x$ if $m=0$, and $x^{\bar y \dots \bar y}$ if $m<0$, where in each case there are $|m|$ factors of $y$ or $\bar y$ in the exponent.

\begin{defn}
\label{defquandle}
A rack $Q$ is called a {\bf quandle} if $x \rhd x = x$ for all $x \in Q$.  Further, if $n \ge 2$ is an integer, then a quandle $Q$ is called an {\bf \boldmath $n$-quandle} if $x^{y^n}=x$ for all $x,y \in Q$.
\end{defn}

\noindent Notice that in an $n$-quandle we also have that $x^{\bar y^n}=x$.  A 2-quandle is also called an {\bf involutory
quandle}. 

Following Fenn and Rourke, we define a presentation $\langle S \, | \, R \, \rangle$ of a rack with generating set $S$ and relations $R$ as a quotient of a free rack. For any set $S$, let $F(S)$ denote the free group on $S$ and in this group let $\bar w$ represent the inverse of the element $w$. 

\begin{defn}
The {\bf free rack} on $S$ is the set of equivalence classes
$$FR(S)=\{\left[a^w \right] \mid a\in S, w \in F(S)\}$$
where $\left[ a^u \right] = \left[ b^v \right]$ if $a=b$ in $S$ and $u=v$ in $F(S)$. The operations in $FR(S)$ are defined by
$\left[a^u \right] \rhd \left[ b^v \right] = \left[a^{u \bar v b v} \right]$ and $\left[ a^u \right] \rhd^{-1} \left[b^v \right]= \left[a^{u \bar v \bar b v} \right]$.
\end{defn}

\noindent From this point on, we will abuse notation and simply let $a^u$ represent the equivalence class $\left[ a^u \right]$. 

A {\bf congruence} on a rack $\mathcal R$ is an equivalence relation $\sim$ that respects the operations. In particular, if $\mathcal{R} = FR(S)$, then a congruence  is a relation with the property that if $a^s \sim b^t$ and $x^u \sim y^v$, then $a^{s \bar u  x u} \sim b^{t \bar v y v}$ and $a^{s \bar u  \bar x u} \sim b^{t \bar v \bar y v}$. Given a congruence on $FR(S)$, then the congruence classes form a quotient of $FR(S)$ that is itself a rack. This notion of a quotient rack allows us to define a rack in terms of generators and relations.

Let $S$ be a finite set of generators and let $R$ be a finite set of relations in $FR(S)$. That is, $R$ is a finite set of ordered pairs of 
the form $(a^u,b)$ where $a,b \in S$ and $u \in F(S)$.  More formally, 
$$R = \{ (a_i^{u_i}, b_i) \mid a_i,b_i \in S, u_i \in F(S), 1\le i \le r\} \subseteq FR(S) \times FR(S).$$
The rack given by the presentation $\langle S \, | \, R \, \rangle$ is then defined to be the quotient of $FR(S)$ by the smallest congruence $\sim_R$ containing $R$. The smallest congruence is described more concretely by Fenn and Rourke in terms of {\em consequences} of the relations in $R$.   Using their work we can derive the following proposition.

\begin{prop} 
\label{submoves} If $\mathcal{R} = \langle S \, | \, R \, \rangle$, then $x^s \sim_R y^t$ if and only if $x^s$ can be taken to $y^t$ by a finite sequence of the following substitutions or their inverses. For all $a,b,c \in S$ and $u,v,w \in F(S)$:
\begin{enumerate}
\setlength\itemsep{.2em}
\item Replace $a^{uw}$ with $a^{uv\bar{v}w}$.
\item If $(a^u,b) \in R$, then replace $a^{uw}$ with $b^{w}$.
\item If $(a^u,b) \in R$, then replace $c^{vw}$ with either
$c^{v\bar{u}au \bar{b} w}$ or $c^{v\bar{u}\bar{a}u bw}$.
\end{enumerate}
\end{prop}

We shall refer to the substitutions in Proposition~\ref{submoves} as {\bf substitution moves}. The proof of the proposition requires showing that the congruence defined by the substitution moves is  the same as the congruence defined by consequences of the relations described in Fenn and Rourke. We leave the details to the interested reader.

\begin{rem} \label{congruence} Notice that since the word $w$ is arbitrary in each of the substitution moves in Proposition~\ref{submoves} it follows that if $x^s \sim_R y^t$, then $x^{sw} \sim_R y^{tw}$ for any $w \in F(S)$.
\end{rem}

As is customary with group presentations, we shall adopt the notation $a^u=b$ to represent a relation $(a^u,b) \in R$ and $x^s=y^t$ to denote $x^s \sim_R y^t$ in the rack $\langle S \mid R \rangle$.  Notice that if all relations $a^a=a$ for $a \in S$ are included in $R$, then $\langle S \, | \, R \, \rangle$ is a quandle. Moreover, for a fixed $n$, if $R$ additionally includes all relations $a^{b^n}=a$ for all distinct $a,b \in S$, then $\langle S \, | \, R \, \rangle$ is an $n$-quandle. To see that this is the case, first notice that if $a^{b^n}=a$, then it follows by substitution move (3) that $x^{\bar b^nab^n \bar a}=x$ for all $x$. Thus, we obtain $y^{a b^n}=y^{b^n a}$ for all $y$ by considering $x = y^{b^n}$. Similarly, $y^{\bar a b^n}=y^{b^n \bar a}$. Since this is true for all generators, it follows by induction that that $y^{w b^n}=y^{b^n w}$ for all words $w \in F(S)$. Now  consider arbitrary elements $x=a^u$ and $y=b^v$. Since the relation $a^{b^n}=a$ has been added and, in the case $a=b$, $a^{a^n}=a$ since $\langle S \, | \, R \rangle$ is a quandle, we have
$$x^{y^n} = a^{u (\bar v b v)^n}= a^{u \bar v b^n v} = a^{b^n u \bar v v} = a^u =x.$$

\section{The rack enumeration process}

In this section, we introduce the notion of a enumeration table and identify important properties of these tables that will remain unchanged during the enumeration process.  Let $\langle S \, | \, R \, \rangle$ be a rack where $S=\{x_1, \dots ,x_g\}$ and $R$ is a set of relations $x_{i_k}^{u_k}=x_{j_k}$ with $u_k$ a reduced word in $F(S)$ for $1 \leq k \leq r$ and $1 \le i_k,j_k \le g$. Let $\bar S = \{\bar x_1, \dots, \bar x_g\}$. Following Winker, we call the relations in $R$ {\bf primary relations}. Notice that for each primary relation $x_{i_k}^{u_k}=x_{j_k}$ and for any $x \in \langle S \mid R \rangle$ we have, by substitution move (3), that 
$$x^{\bar{u}_kx_{i_k}u_k\bar{x}_{j_k}}=x.$$

\noindent These relations are called {\bf secondary relations} by Winker. The word $\bar{u}_kx_{i_k}u_k\bar{x}_{j_k}$ may not be reduced, in which case, we will use its reduced form in the procedure. We denote the set of reduced secondary relations by $R_2$.

\begin{defn} An {\bf enumeration table $\mathcal{E}$} for a rack $\mathcal{R}=\langle S \, | \,R \, \rangle$ is a 4-tuple $(\omega,A,\tau,\rho)$ where $\omega$ is the number of rows in the table, $A$ is a partial function from $\{1,2,...,\omega\} \times (S \cup \bar{S})$ to $\{1,2,...,\omega\}$, $\tau: \{1,2,...,\omega\} \rightarrow \mathcal{R}$ is a function, and $\rho:\{1,2,...,\omega\} \rightarrow \{1,2,...,\omega\}$ is a function with the property that $\rho(i) \leq i$ for all $1 \leq i \leq \omega$.
\end{defn}

We will denote $A(i,y)$ by $i^y$, and so $i^y$ may  or may not be defined since $A$ is a partial function. Define the {\bf live elements} of $\mathcal E$ to be the set $\Omega=\{i  \mid 1 \le i \le \omega \mbox{ and } \rho(i)=i \}$ and call $\mathcal E$ {\bf complete} if for every $i \in \Omega$ and for every $y \in S \cup \bar S$, we have that $i^x$ is defined.

If $\mathcal R$ is a finite rack, we describe a process that produces a sequence of tables $\mathcal{E}_0, \mathcal{E}_1,...,\mathcal{E}_f$ so that $\mathcal{E}_f$ is complete,  $\Omega$ is a rack with operations provided by $\mathcal E_f$, and $\tau : \Omega \rightarrow \mathcal R$ is an isomorphism.  In our description of the process we will represent a enumeration table by a rectangular array whose rows are numbered $1$ through $\omega$ and whose columns are labelled by the elements of $S \cup \bar S$, $\tau$, and $\rho$. The entry in row $i$ and column $y \in S \cup \bar S$ is $i^y$ if it is defined and empty otherwise. The last two columns give values of $\tau$ and $\rho$, respectively, for the row label $i$.  We begin with an example that illustrates the process before giving the details of the algorithms involved.

Consider the rack with presentation 
$$\mathcal R = \langle a,b \mid a^{ba}=b, b^{ba}=a, a^{bb}=a, b^{aa}=b \rangle.$$
We initialize the enumeration table $\mathcal E$ by letting $1$ represent the element $a$ and $2$ represent $b$.  That is, we define $\tau(1)=a$ and $\tau(2)=b$ and we set $\rho(1)=1$ and $\rho(2)=2$. We next find the set $R_2$ of secondary relations. For each secondary relation $x^w=x$ we record the reduced word $w$.
$$R_2=\{ \bar a \bar b a b a \bar b, \bar a b, \bar b \bar b a b b \bar a, \bar a \bar a b a a \bar b \}.$$
\begin{center}
\begin{small}
\begin{tabular}{ c|cccccc } 
$\mathcal E$ & $a$ & $b$ & $\bar{a}$ & $\bar{b}$ & $\tau$ & $\rho$\\ \hline
1 &  &  &  & & $a$ & 1\\ 
2 &  &  &  &  & $b$ & 2
\end{tabular}
\end{small}
\end{center}
\medskip
The next step is to encode information from the primary relations.  Consider the first relation $a^{ba}=b$.  Since $\tau(1)=a$ and $\tau(2)=b$, we would like our table to satisfy $1^{ba}=2$.  However, $1^b$ is not defined in $\mathcal E$ so we {\em define} a new element $3=1^b$ and extend the map $\tau$ so that $\tau(3) = \tau(1)^b = a^b$. Since $1^b=3$ we also add the {\em inverse entry} $3^{\bar b}=1$. We indicate where a definition is made by underlining the defined entry in the table. 

Notice that $1^{ba}=2$ and $1^b=3$ imply that $3^a =2$. This is called a {\em deduction} and we also encode it, and its inverse entry $2^{\bar a}=3$,  in our table. This is called {\em scanning} the first primary relation and the process is illustrated by a {\em helper table} shown to the right of the enumeration table below. Parentheses in the helper table indicate where {\em forward} and {\em backward scanning} end. The deduction $3^a=2$ occurs where open parentheses meet.  The corresponding entries added to $\mathcal E$ are enclosed with open parentheses. 
\medskip
\begin{center}
\begin{small}
\begin{tabular}{c|cccccc} 
$\mathcal E$ & $a$ & $b$ & $\bar{a}$ & $\bar{b}$ & $\tau$ & $\rho$\\ \hline
1 &  & \underline 3 &  & & $a$ & 1\\ 
2 &  &  & (3) &  & $b$ & 2 \\
3 & (2) &  &  & \underline 1 & $a^b$ & 3
\end{tabular}
\hspace{.25in}
\begin{tabular}{ccc} 
 & $b$ & $a$ \\ \hline
1 & \underline 3) &   (2 
\end{tabular}
\end{small}
\end{center}
\medskip
Scanning the remaining primary relations gives another definition and three additional deductions. Again we mark the modified entries in the table to indicate whether they came from a definition or a deduction and include the helper tables for the relations.
\medskip
\begin{center}
\begin{small}
\begin{tabular}{c|cccccc} 
$\mathcal E$ & $a$ & $b$ & $\bar{a}$ & $\bar{b}$ & $\tau$ & $\rho$\\ \hline
1 &  & 3 & (4) & (3) & $a$ & 1\\ 
2 & (3) & \underline 4 & 3 &  & $b$ & 2 \\
3 & 2 & (1) & (2)  & 1 & $a^b$ & 3 \\
4 & (1) & & & \underline 2 & $b^b$ & 4
\end{tabular}
\hspace{.25in}
\begin{tabular}{ccc} 
 & $b$ & $a$ \\ \hline
2 &  \underline 4) &  (1 \\
\\
  & $b$ & $b$ \\ \hline
1 &  3) &  (1 \\
\\
  & $a$ & $a$ \\ \hline
2) &  (3 &  2 
\end{tabular}
\end{small}
\end{center}
\medskip
The table $\mathcal E$ above represents the conclusion of a definite loop in the rack enumeration process that scans all primary relations. The next loop in the process scans each secondary relation for each live row. Since scanning may introduce new live rows, this loop is indefinite.

Consider scanning the first secondary relation $\bar a \bar b a b a \bar b$ for live row $1$. A helper table for this scan is shown below.  Notice that scanning forward from $1$ we have  $1^{\bar a \bar b a b}=1$ is defined but we cannot scan forward further because $1^a$ is not defined. So we begin scanning backwards from $1$.  In doing so we have that $1^{b \bar a}=2$ is defined and  we have arrived at a {\em coincidence} where two different values, $1$ and $2$, appear in the same location in the helper table (which we denote by $[1_2]$). 
\medskip
\begin{center}
\begin{small}
\begin{tabular}{ccccccc} 
 & $\bar a$ & $\bar b$ & $a$ & $b$ & $a$ & $\bar b$ \\ \hline
1 & 4 & 2 & 3 & $[1_2]$ & 3 & 1
\end{tabular}
\end{small}
\end{center}
\medskip
To resolve this coincidence, we eliminate the larger index $2$ and merge any data from row $2$ into row $1$ of the table. We do this by first changing the value of $\rho(2)$ to be 1, which indicates that 2 is a dead row and that all occurrences of 2 will eventually be replaced. Then for each $x \in \{a ,b, \bar a, \bar b \}$ we do one of three things. If $2^x$ is undefined, we proceed to the next value for $x$. If $2^x=i$ and $1^x$ is undefined, then we remove $2^x=i$ and $i^{\bar x}=2$ from $\mathcal E$  and add $1^x=i$ and $i^{\bar x}=1$ to $\mathcal E$. Notice this situation occurs for $x=a$.  Otherwise, if $2^x=i$ and $1^x=j$ then we also remove $2^x=i$ and $i^{\bar x}=2$ but, instead of adding new entries, we queue up a new coincidence between $i$ and $j$.  Notice this situation occurs for $x=b$ and $x=\bar a$ and, in both cases, the new coincidence is $[4_3]$. After resolving the coincidence $[4_3]$ in the same manner, no new coincidences appear and the resulting table is shown below.  The entries in the table changed by the coincidences are marked by closed parentheses and rows $2$ and $4$ are now dead rows.
\medskip
\begin{center}
\begin{small}
\begin{tabular}{c|cccccc} 
$\mathcal E$ & $a$ & $b$ & $\bar{a}$ & $\bar{b}$ & $\tau$ & $\rho$\\ \hline
1 & [3] & 3 & [3] & 3 & $a$ & 1\\ 
2 & $\not \! 3$ & $\not \! 4$ & $\not \! 3$ &  & $b$ & 1 \\
3 & [1] & 1 & [1]  & 1 & $a^b$ & 3 \\
4 & $\not \! 1$ & & & $\not \! 2$ & $b^b$ & 3
\end{tabular}
\end{small}
\end{center}
\medskip
At this stage the enumeration table is complete but the process is not. Continuing to scan the secondary relations will never lead to a new definition or deduction, however, there could be additional coincidences. The reader can verify that all remaining scans {\em complete correctly}, that is, forward scanning reaches the end of the relation without any definitions, deductions, or coincidences needed.  It now follows, as we show later, that the rack is of order 2 with
$\mathcal R = \{ a, a^b\}$.
Moreover, a multiplication table for the rack can now be derived from the complete enumeration table using the rack axioms.

In Algorithm~\ref{enumeration}, we present pseudocode for the rack enumeration process described in the example. The pseudocode contains several  subroutines that will be defined subsequently. Since the process contains an indefinite loop, a run limit is used to guarantee that the process terminates. We say that the {\sc Enumerate} process {\bf completes} when it returns a complete table in line~\ref{complete}. Specifically note that, if the process completes, then all secondary relations have been scanned for all live rows.

\begin{algorithm}[H]
\caption{The rack enumeration process}
\begin{algorithmic}[1]
\Procedure{Enumerate}{$S,R,M$}
\State Input: generators $S$, primary relations $R$, run limit $M$
\State $(R_2, {\mathcal T}):= \mbox{\sc Init}(S, R)$ \Comment{derive $R_2$, initialize table}
\For {$x_{i}^{u}=x_{j} \in R$} \Comment{scan primary relations}
\State {\sc Scan}$(\sim \! \mathcal T,\mbox{\sc Rep}(i), u ,\mbox{\sc Rep}(j))$ 
\EndFor
\State $i:=1$
\While {$i \le \mbox{max}( \Omega)$ and $i \le M$}
\For {$w \in R_2$} \Comment{scan secondary relations}
\If {$i \in \Omega$}
\State {\sc Scan}$(\sim \! \mathcal T,i,w,i)$ 
\Else 
\State break \Comment{$i$ is dead, stop scanning}
\EndIf
\EndFor
\If {$i \in \Omega$}
\For {$y \in S \cup \bar S$ and $i^y$ undefined}
\State {\sc Define}$(\sim \! \mathcal T, i,y)$\label{ward} \Comment{fill undefined entries in row $i$}
\EndFor
\EndIf
\State $i:=i+1$
\EndWhile
\If {$i > \mbox{max}( \Omega)$}
\State \Return $\mathcal E$ \Comment{process completes}\label{complete}
\Else 
\State \Return run limit exceeded
\EndIf
\EndProcedure
\end{algorithmic}
\label{enumeration}
\end{algorithm}

Before we describe the subroutines called by {\sc Enumerate}, we list five properties which we will show to be true after the enumeration table is initialized and which remain true after each step of the procedure. These properties will then allow us  to produce the rack isomorphism $\tau:\Omega \rightarrow \mathcal R$ when the process completes.  First, we need some additional definitions. Let $w=y_1y_2...y_t \in F(S)$  and let $j \in \{1,2,\dots, \omega\}$.  We say $j^w$ is {\bf defined and equal} to $k$ if $j_0=j$ and for $1 \leq i \leq t$ we have $j_i=j_{i-1}^{y_i}$ is defined and $j_t=k$.
As seen in the example, the function $\rho$ will be used to record when coincidences occur. Let $\mbox{orbit}(i) = \{ \rho^t(i) \mid t \ge 0\}$ where $\rho^t$ is $\rho$ composed with itself $t$ times. Define the {\bf least representative} of $i$ by $\mbox{\sc Rep}(i)= \min(\mbox{orbit}(i))$. \\

\noindent{\bf Property 1.} $1 \in \Omega$ and $\tau(i)=x_i$ for all $1 \le i \le g$.\\
\vskip -.1 in
\noindent{\bf Property 2.} If $i, j \in \{1,2,\dots,\omega\}$ and $y \in S \cup \bar{S}$, then $i^y=j$ if and only if $j^{\bar{y}}=i$.\\ 
\vskip -.1 in
\noindent{\bf Property 3.} If $i, j \in \{1,2,\dots,\omega\}$, $y \in S \cup \bar{S}$, and $i^y=j$, then $\tau(i)^y=\tau(j)$ in $\mathcal{R}$.\\
\vskip -.1 in
\noindent{\bf Property 4.} If $j \in \Omega$, then there exists $i \in \Omega$, $1\le i \le g$,  and $w \in F(S)$ such that $j=i^w$.\\
\vskip -.1 in
\noindent{\bf Property 5.} If $i \in \{1,2,\dots,\omega\}$, then $\tau(i) = \tau(\mbox{\sc Rep}(i))$ in $\mathcal R$.\\

Notice that the single element $y \in S \cup \bar{S}$ in both Properties 2 and 3 can be replaced by any word $w\in F(S)$. This is easily proven by inducting on the length of $w$.

Next, we introduce Algorithms~\ref{init} and \ref{define}. The first initializes the enumeration table and produces a set of reduced secondary relations. The second creates a new row in the table and two new entries.  The notation $\sim \! \mathcal T$ in the argument list of {\sc Define} (and already appearing in Algorithm~\ref{enumeration}) means that the procedure changes $\mathcal E$. We adopt this convention throughout.

\vskip 0 in
\begin{algorithm}[H]
\caption{Initializing the table}
\begin{algorithmic}[1]
\Procedure{Init}{$S,R$}
\State  Input: generators $S$, primary relations $R$
\State $\omega := g;\ A := \phi;\ R_2:=\phi$
\For {$x_i^u =x_j \in R$}\Comment{derive secondary relations $R_2$}
\State $w := \bar u x_i u \bar x_j$ (reduced)
\State $R_2:= R_2 \cup \{w\}$ 
\EndFor
\For {$1 \le i \le \omega$}
\State $\tau(i):=x_i$
\State $\rho(i):=i$
\EndFor
\State $\mathcal T := (\omega, A, \tau, \rho)$
\State \Return $(R_2, \mathcal T)$ 
\EndProcedure
\end{algorithmic}
\label{init}
\end{algorithm}

\begin{algorithm}[H]
\caption{Defining $i^y$}\begin{algorithmic}[1]
\Procedure{Define}{$\sim \! \mathcal T,i, y$}
\State  Input: $\mathcal{T}$, $i \in \Omega$, $y \in S \cup \bar{S}$
\State $\omega:=\omega+1$ \Comment{add new row to table}
\State $i^y:=\omega$; $\omega^{\bar{y}}:=i$
\State $\tau(\omega) := \tau(i)^y$
\State $\rho(\omega) :=\omega$
\EndProcedure
\end{algorithmic}
\label{define}
\end{algorithm}

It is straightforward to show the following. 

\begin{prop}
\label{Definelemma}  Properties 1--5 are true after calling {\sc Init} and they are preserved by each call to {\sc Define}. 
\end{prop}

By saying a call to {\sc Define} preserves the properties, we mean  that if they are true before a call to {\sc Define}, then they remain true after the call. We next define the procedure {\sc Scan} in Algorithm~\ref{scan}. It will call on {\sc Define} and the additional routines {\sc Deduction} and {\sc Coincidence}, that will be given in Algorithms~\ref{deduction} and \ref{coincidence}, respectively. 

\begin{algorithm}[H]
\caption{Scanning the relation $i^w = j$}
\begin{algorithmic}[1]
\Procedure{Scan}{$\sim \! \mathcal T, i,w,j$}
\State Input: $\mathcal{T}$, $i, j \in \Omega$, $w=y_1y_2...y_t \in F(S), reduced$
\State $f:=1;\ b:=t$; \Comment{initialize forward and backward counters}
\State $k:=i ;\ \ell:=j$; \Comment{initialize forward and backward scans}
\While {$f\le b$}
\While {$f \le b$ and $k^{y_f}$ defined}\label{scf}\Comment{scan forward}\label{scf}
\State $k:=k^{y_f};\ f:=f+1$ 
\EndWhile
\While {$f \le b$ and $\ell^{\bar{y}_b}$ defined}\Comment{scan backward}
\State $\ell := \ell^{\bar{y}_b};\ b:=b-1$ 
\EndWhile
\If {$f <b$}
\State {\sc Define}($\sim \! \mathcal T,k, y_f$)\Comment{extend forward scan}\label{sfdefine}
\ElsIf{$f=b$}
\State {\sc Deduction}($\sim \! \mathcal T,k, y_f, \ell$) \Comment{scans meet}\label{sfded}
\State break \Comment{break from while loop}
\ElsIf{$k \neq \ell$}\Comment{$b<f$}
\State {\sc Coincidence}($\sim \! \mathcal T, k,\ell$)\Comment{scans overlap incorrectly}\label{sfcoin}
\Else
\State break\Comment{scan completes correctly}
\EndIf
\EndWhile
\EndProcedure
\end{algorithmic}
\label{scan}
\end{algorithm}

The {\sc Scan} procedure  scans forward as far as possible and then scans backward as far as possible. After doing so, if there is a gap, then a definition is made and the cycle is repeated until the scans meet or overlap. This leads to a {\sc Deduction} or {\sc Coincidence}, respectively.  Furthermore, it is not difficult to prove that because the word $w$ is reduced, if {\sc Define} is called, then the procedure ends with a call to {\sc Deduction}.

\begin{algorithm}[H]
\caption{Making the deduction $i^y = j$}
\begin{algorithmic}[1]
\Procedure{Deduction}{$\sim \! \mathcal T, i, y ,j$}
\State Input: $\mathcal{T}$, $i,j \in \Omega$, $y \in S \cup \bar S$
\State $i^y := j$; $j^{\bar y} := i$ \label{dedadd}
\EndProcedure
\end{algorithmic}
\label{deduction}
\end{algorithm}

In order to see that a call to {\sc Scan} preserves Properties~1--5, it suffices to show that each call to the subroutines {\sc Deduction} and {\sc Coincidence} preserves the properties.

\begin{prop}
\label{Deductionlemma} Properties 1--5 are preserved by each call to {\sc Deduction}.
\end{prop}

\begin{proof} Since no new rows are added and no values of $\tau$ and $\rho$ are changed by {\sc Deduction}, Properties 1, 4, and 5 are clearly preserved.  Property~2 is preserved since {\sc Deduction} adds both $i^y=j$ and $j^{\bar y}=i$ to the table.

We now discuss Property~3. We need only consider the case where $k^y$ is undefined before the call to {\sc Deduction} and $k^y=l$ after the call. Suppose this occurred from a call to {\sc Scan}($i, y_1 y_2 \dots y_t, j$). Then, the $f=b$, $i^{y_1 \dots y_{f-1}} = k$, $k^{y_f}$ is not defined, $j^{\bar y_t \dots \bar y_{b+1}}=\ell$, and $\ell^{\bar y_b}$ is not defined. The deduction adds two new entries $k^{y_f}=\ell$ and $\ell^{\bar y_f}=k$ to the table, so we must prove $\tau(k)^{y_f}=\tau(\ell)$ and $\tau(\ell)^{\bar y_f} = \tau(k)$. 

Because  Properties~1-5 were satisfied up to this call to {\sc Deduction}, we have
that $\tau(k) = \tau(i)^{y_1 \dots y_{f-1}}$ and  $\tau(\ell) = \tau(j)^{\bar y_t \dots \bar y_{f+1}}$. Now there are two cases depending on whether the scan was applied to a primary or secondary relation. 

\medskip

\noindent Case (1). Assume $1 \leq i, j \leq g$ and $x_i^{y_1 \dots y_t}=x_j$ is a primary relation.  With the notation above, we have
\begin{align*}
\tau(k) &= \tau(i)^{y_1 \dots y_{f-1}}\\
\tau(k)^{y_f} &=\tau(i)^{y_1 \dots y_f} & & \mbox{Remark~\ref{congruence}}\\
&=x_i^{y_1\dots y_f} & & \mbox{from {\sc Init}$(S,R)$}\\
&=x_i^{y_1...y_f y_{f+1}\dots y_t \bar{y}_t \dots \bar y_{f+1}} &&\mbox{Prop.~\ref{submoves} (1)}\\
&=x_j^{\bar{y}_t...\bar{y}_{f+1}} &&\mbox{Prop.~\ref{submoves} (2)}\\
&=\tau(j)^{\bar{y}_t...\bar{y}_{f+1}} & & \mbox{from {\sc Init}$(S,R)$}\\
&=\tau(\ell).
\end{align*}

\noindent It follows from Remark~\ref{congruence} and Proposition~\ref{submoves} (1) that $\tau(\ell)^{\bar y_f} = \tau(k)$ as well.\\

\noindent Case (2). Assume $i=j$ and $y_1 \dots y_t \in R_2$.  We now have
\begin{align*}
\tau(k) &= \tau(i)^{y_1 \dots y_{f-1}}\\
\tau(k)^{y_f} &=\tau(i)^{y_1 \dots y_f} & & \mbox{Remark~\ref{congruence}}\\
&=\tau(i)^{y_1...y_f y_{f+1}\dots y_t \bar{y}_t \dots \bar y_{f+1}} &&\mbox{Prop.~\ref{submoves} (1)}\\
&=\tau(i)^{\bar{y}_t...\bar{y}_{f+1}} &&\mbox{Prop.~\ref{submoves} (3)}\\
&=\tau(\ell).
\end{align*}
\noindent As in Case (1), this implies $\tau(\ell)^{\bar y_f} = \tau(k)$ as well.
\end{proof}

Before giving the {\sc Coincidence} procedure we describe three additional routines {\sc Merge}, {\sc Rep}, and {\sc Update} which will all be used by {\sc Coincidence}.  The procedure {\sc Rep}($i$) finds the least representative of $i$ and the related procedure {\sc Update}$(i)$ changes $\mathcal E$ so that $\rho(j)= \mbox{\sc Rep}(i)$ for all $j \in \mbox{orb}(i)$. Notice that $\rho(i) \le i$ is required for the procedure {\sc Rep} to find the least representative.

\begin{algorithm}[H]
\caption{Finding the least representative of $i$}
\begin{algorithmic}[1]
\Procedure{Rep}{$i$}
\State Input: $\mathcal E$, $i \in \{1,2,\dots,\omega\}$
\State $j := i$
\While {$\rho(j) < j$}
\State $j := \rho(j)$
\EndWhile
\State \Return $j$
\EndProcedure
\end{algorithmic}
\end{algorithm}

\begin{algorithm}[H]
\caption{Setting $\rho(j) = \mbox{\sc Rep}(i)$ for all $j \in \mbox{orb}(i)$}
\begin{algorithmic}[1]
\Procedure{Update}{$\sim \! \mathcal T, i$}
\State Input: $\mathcal E$, $i \in \{1,2,\dots,\omega\}$
\State $\epsilon = \mbox{\sc Rep}(i)$; $n:=i$; $m:=\rho(n)$
\While {$m < n$}
\State $\rho(n):= \epsilon$; $n:=m$; $m:=\rho(n)$
\EndWhile
\EndProcedure
\end{algorithmic}
\end{algorithm}

The {\sc Coincidence} procedure is called when scanning forward and backward produce two distinct values $k$ and $\ell$ in the same location of the helper table. The procedure changes $\rho$ of the larger of the two values, replaces all occurrences of the larger value with its new smallest representative, and merges information from the larger value's row into the row for its smallest  representative. Sometimes the merging of rows will introduce new coincidences.  Hence, our procedure must produce a queue of coincidences that will be resolved in order. The {\sc Merge} procedure in Algorithm~\ref{merge} adds to the queue of coincidences and changes values of $\rho$ to record which elements are to be killed.  

\begin{algorithm}[H]
\caption{Adding a coincidence to the merge queue}
\begin{algorithmic}[1]
\Procedure{Merge}{$\sim \! Q,\sim \! \rho, m, n$}
\State Input: $\rho$, $Q$ a queue of coincidences, $m \equiv n$ a coincidence  
\State $\mu:=\mbox{\sc Rep}(m)$;  $\nu:=\mbox{\sc Rep}(n)$
\If {$\mu \neq \nu$}
\State append $\mbox{max}(\mu,\nu)$ to $Q$
\State $\rho(\mbox{max}(\mu,\nu)):=\mbox{min}(\mu,\nu)$ 
\EndIf
\EndProcedure
\end{algorithmic}
\label{merge}
\end{algorithm}

Notice that only live elements are added to the queue but then are immediately killed. This means that the queue is always a distinct set of dead elements.  Notice also that $\mbox{\sc Rep}(m) = \mbox{\sc Rep}(n)$ after a call to {\sc Merge}$(m,n)$. We are now prepared to define {\sc Coincidence}.

\begin{algorithm}[H]
\caption{Resolving a coincidence $m \equiv n$}
\begin{algorithmic}[1]
\Procedure{Coincidence}{$\sim \! \mathcal T, m, n$}
\State Input: $\mathcal{T}$, $m,n \in \Omega$, $m \equiv n$ a coincidence
\State $Q:=\phi$
\State {\sc Merge}($\sim \! Q, \sim \! \rho,m, n$)\Comment{queue coincidence $m \equiv n$}\label{merge1}
\State $q:=1$
\While{$q \le \mbox{length}(Q)$}\label{whileq}
\State $d := Q(q)$; $q:=q+1$\Comment{take $q^{\rm th}$ element off $Q$}
\For{$x \in S \cup \bar S$}\label{forx}
\If{$d^x=e$}\label{dxe}
\State undefine $d^x$ and $e^{\bar{x}}$\Comment{remove inverse pair}\label{coinrem}
\State $\delta:=\mbox{\sc Rep}(d)$;\  {\sc Update}$(d)$
\State $\epsilon:= \mbox{\sc Rep}(e)$; {\sc Update}$(e)$
\If{$\delta^x$ is defined}\label{inductline}
\State {\sc Merge}($\sim \!Q,\sim \! \rho, \epsilon, \delta^x$)\label{merge2}\Comment{queue new coincidence}
\ElsIf{$\epsilon^{\bar{x}}$ is defined}
\State {\sc Merge}($\sim \!Q,\sim \! \rho,\delta,\epsilon^{\bar x}$)\label{merge3}\Comment{queue new coincidence}
\Else 
\State $\delta^x := \epsilon$; $\epsilon^{\bar x}:=\delta$ \Comment{add inverse pair}\label{coinadd}
\EndIf 
\EndIf \label{endif}
\EndFor
\EndWhile
\EndProcedure
\end{algorithmic}
\label{coincidence}
\end{algorithm}

The {\sc Coincidence} procedure involves an indefinite  loop since the length of $Q$ can increase. However, since $Q$ is a finite list of distinct elements from $\{1,2,\dots,\omega\}$ and since no process in {\sc Coincidence} changes the number of rows $\omega$ of $\mathcal E$, the loop will terminate.
We need the following lemma to prove that {\sc Coincidence} preserves Properties~1--5.

\begin{lem} If $i^y=j$ where $i,j \in \Omega$ and $y \in S \cup \overline S$ before a call to {\sc Coincidence}, then $\mbox{\sc Rep}(i)^y = \mbox{\sc Rep}(j)$ after the call.
\label{repprop}
\end{lem}

{\bf Proof.}  Notice that Property~2 remains true after a call to {\sc Coincidence} because entries in $\mathcal E$ are only removed or added in inverse pairs by the procedure. The {\sc Coincidence} routine incrementally builds a queue of distinct elements from $\{1,2,\dots, \omega\}$ that are all dead and have been removed from $\mathcal E$ by the time the procedure has completed. Let $Q'$ denote the final queue created by {\sc Coincidence}. The initial call to {\sc Merge} in line~4 initializes the queue by adding either $m$ or $n$ to it.

Assume $i^y=j$ before a call to {\sc Coincidence}.  If $i,j \not\in Q'$, then $i^y=j$ is not removed from $\mathcal E$ and after the call we have $i=\mbox{\sc Rep}(i)$ and $j=\mbox{\sc Rep}(j)$. Therefore, $\mbox{\sc Rep}(i)^y=\mbox{\sc Rep}(j)$ after the call. So assume then that $i$ or $j$ is in $Q'$. We will show that after executing lines~\ref{dxe}--\ref{endif}, there exists $p$ and $q$ such that $p^y=q$, $q^{\bar y}=p$, $\mbox{\sc Rep}(p) = \mbox{\sc Rep}(i)$, and $\mbox{\sc Rep}(q)=\mbox{\sc Rep}(j)$.

Assume first that $i^y=j$,  $i \in Q'$, and if $j \in Q'$ then $j$ appears after $i$ in the queue. We leave the other case to the reader. Since $j$ does not appear before $i$ in $Q'$, there is a point in the execution of the procedure where we reach line~\ref{dxe} with $d=i$, $x=y$, and $e=j$.  Starting at line~\ref{coinrem}, first $i^y=j$ and $j^{\bar y}=i$ are removed from the table and then
$\delta = \mbox{\sc Rep}(i)$ and $\epsilon = \mbox{\sc Rep}(j)$ are defined. There are three cases to consider.

\begin{enumerate}
\item If $\delta^y = f$, then $f^{\bar y}=\delta$ (since Property~2 is satisfied) and a call to $\mbox{\sc Merge}(\epsilon, f)$ is made.  After this call we have $\mbox{\sc Rep}(f) = \epsilon =\mbox{\sc Rep}(j)$ and $\mbox{\sc Rep}(i)=\delta$. If we now define $p = \delta$ and $q= f$, then $p^y=q$, $q^{\bar y}=p$, $\mbox{\sc Rep}(p) = \mbox{\sc Rep}(i)$, and $\mbox{\sc Rep}(q)=\mbox{\sc Rep}(j)$.  
\item If $\delta^y$ is undefined but $\epsilon^{\bar y} =f$, then $f^y=\epsilon$ and a call to $\mbox{\sc Merge}(\delta, f)$ is made.  Similar to above, we have $\mbox{\sc Rep}(f)=\delta=\mbox{\sc Rep}(i)$ and $\epsilon=\mbox{\sc Rep}(j)$ after the merge. In this case, define $p=f$ and $q = \epsilon$ and the result is true.
\item If $\delta^y$ and $\epsilon^{\bar y}$ are both undefined, then we add the entries $\delta^y=\epsilon$ and $\epsilon^{\bar y}=\delta$ to $\mathcal E$. Since no values of $\rho$ are changed in this case, we still have that
$\mbox{\sc Rep}(i)=\delta$ and $\mbox{\sc Rep}(j)=\epsilon$. In this case, define $p=\delta$ and $q=\epsilon$ and the result is true.
\end{enumerate}
Notice that in every case, neither $p$ nor $q$ can appear before $i$ in $Q'$.

We are now prepared to prove that if $i^y=j$ before a call to {\sc Coincidence}, then $\mbox{\sc Rep}(i)^y = \mbox{\sc Rep}(j)$ after the call.  Set $i_0=i$ and $j_0=j$. If $i_0,j_0 \not\in Q'$, then we are done. Otherwise, as shown above, there exists $i_1,j_1$ such that $i_1^y=j_1$, $\mbox{\sc Rep}(i_{0})=\mbox{\sc Rep}(i_1)$, and $\mbox{\sc Rep}(j_{0})=\mbox{\sc Rep}(j_1)$. If $i_1,j_1 \not\in Q'$, then we are done. Otherwise, note that the first occurrence of either $i_0$ or $j_0$ in $Q'$ must precede the first occurrence of $i_1$ or $j_1$ by our remark above. Since $Q'$ is finite, this implies that the process must terminate with a last equation $i_{\ell}^y= j_{\ell}$ where $i_\ell,j_\ell \not\in Q'$. Therefore,
$$\mbox{\sc Rep}(i)^y = \mbox{\sc Rep}(i_\ell)^y=i_\ell^y = j_\ell=\mbox{\sc Rep}(j_\ell)=\mbox{\sc Rep}(j).$$
 \hfill $\Box$
  
We are now prepared to prove that {\sc Coincidence} preserves Properties 1--5.

\begin{prop}
\label{Coincidencelemma}
Properties 1--5 are preserved by each call to {\sc Coincidence}.
\end{prop} 

\begin{proof} Notice that $1 \in \Omega$ after a call to {\sc Coincidence} because {\sc Merge} will never change $\rho(1)$. Furthermore, none of the procedures alter $\tau$ so Property~1 remains true. As already seen in the proof of Lemma~\ref{repprop},  Property~2 remains true after a call to {\sc Coincidence}. 

It is convenient to prove Properties~3 and 5 together. Assume both properties are true before a call to {\sc Coincidence}. We first consider the initial call to {\sc Merge} on line~\ref{merge1}. The {\sc Merge} procedure does not change values of $A$ nor $\tau$ and so Property~3 is still true after line~\ref{merge1}. On the other hand, {\sc Merge} does change values of $\rho$ and so we must show Property~5 remains true after line~\ref{merge1}. There are two cases to consider depending on whether {\sc Coincidence} was called by {\sc Scan} when considering a primary or a secondary relation. First consider {\sc Scan}$(i,y_1 \dots y_t,j)$ with $1 \le i,j \le g$ and $x_i^{y_1 \dots y_t}=x_j \in R$ a primary relation. A coincidence occurs when $i^{y_1 \dots y_{f-1}}=k$ is defined, $j^{\bar y_t \dots \bar y_f}=\ell$ is defined, $k \neq \ell$, and $k, \ell \in \Omega$. Hence, $\mbox{\sc Rep}(k)=k$ and $\mbox{\sc Rep}(\ell)=\ell$ before the call to {\sc Coincidence}.   Since Property~3 is true before the call, we have that $\tau(k)=x_i^{y_1 \dots y_{f-1}}$ and $\tau(\ell)=x_j^{\bar y_t \dots \bar y_f}$. Therefore,
\begin{align*}
\tau(k) &=x_i^{y_1 \dots y_{f-1}}\\
&=x_i^{y_1 \dots y_{f-1} y_f \dots y_t \bar y_t \dots \bar y_f} &&\mbox{Prop.~\ref{submoves} (1)}\\ 
&=x_j^{\bar y_t \dots \bar y_f} &&\mbox{Prop.~\ref{submoves} (2)}\\
&=\tau(\ell).
\end{align*}
Assume that $k > \ell$ and so after the call to {\sc Merge}$(k,\ell)$ in line~\ref{merge1} we have $\mbox{\sc Rep}(k) = \mbox{\sc Rep}(\ell) = \ell$.  Therefore, $\tau(\ell) = \tau( \mbox{\sc Rep}(\ell))$ and $\tau(k) = \tau(\ell) = \tau(\mbox{\sc Rep}(k))$ so Property~5 remains true after line~\ref{merge1}. The case where $k < \ell$ is similar as is the case when {\sc Scan} is applied to a secondary relation.

We next show that the two properties are preserved by inducting on the number of times lines~\ref{merge1} and \ref{inductline} are executed. The previous argument for the call to {\sc Merge} in line~\ref{merge1} establishes the base case. So assume that both properties are true and we arrive at  line \ref{inductline} with $d \in Q$, $x \in S \cup \bar S$, $d^x=e$, $\delta = \mbox{\sc Rep}(d)$, and $\epsilon = \mbox{\sc Rep}(e)$.  Thus, at this point,  $\tau(d)^x=\tau(e)$, $\tau(d) =\tau(\delta)$, and $\tau(e)=\tau(\epsilon)$  by our inductive hypothesis. There are three cases to consider: we may call {\sc Merge} on lines \ref{merge2} or \ref{merge3}, or add two entries to the table on line 18.  In no case are values of $\tau$ changed, but {\sc Merge} may changes values of $\rho$. Thus, Property 5 will remain true if line \ref{coinadd} is executed and we must still show that it remains true if line \ref{merge2} or \ref{merge3} is executed. Similarly, Property 3 will remain true if line \ref{merge2} or \ref{merge3} is executed and we must still show that it remains true if line  \ref{coinadd} is executed.

Suppose $\delta^x=f$ and we call {\sc Merge}$(\epsilon,f)$ in line~\ref{merge2}. So, by Property~3, $\tau(\delta)^x=\tau(f)$. Suppose now that $\ell$ is arbitrary. We want to show that $\tau(\ell)=\tau(\mbox{\sc Rep}(\ell))$ after the call to  {\sc Merge}$(\epsilon,f)$. Assume that $\epsilon > \mbox{\sc Rep}(f)=\phi$ in which case {\sc Merge} will set $\rho(\epsilon)=\phi.$ If before the call to {\sc Merge}, $\mbox{\sc Rep}(\ell)\ne \epsilon$ then $\mbox{\sc Rep}(\ell)$ will be unchanged and after the call we will still have $\tau(\ell)=\tau(\mbox{\sc Rep}(\ell))$. However, if $\mbox{\sc Rep}(\ell)= \epsilon$ before the call, then $\tau(\ell) = \tau(\epsilon)$ and, after the call, we will have $\mbox{\sc Rep}(\ell)=\phi$. Using all of the above, we have
$$\tau(\ell)=\tau(\epsilon)=\tau(e)=\tau(d)^x=\tau(\delta)^x=\tau(f)=\tau(\phi)=\tau(\mbox{\sc Rep}(\ell)).$$
The case where $\epsilon < \phi$ is similar.  If $\delta^x$ is undefined and $\epsilon^{\bar x}$ is defined, then the argument is similar.

Consider now the third possibility where $\delta^x$ and $\epsilon^{\bar x}$ are both undefined. In this case, two entries $\delta^x=\epsilon$ and $\epsilon^{\bar x}=\delta$ are added to $\mathcal E$ by line~\ref{coinadd}. By the inductive hypotheses, we have 
$$\tau(\delta)^x = \tau(d)^x =\tau(e)=\tau(\epsilon).$$
Therefore, Property~3 remains true after executing line~\ref{coinadd}. 

Finally, consider Property~4. If $j \in \Omega$ after the call to {\sc Coincidence}, then before the call, $j \in \Omega$ and there exists $i \in \Omega \cap \{1,2,\dots, g\}$ and $w \in F(S)$ such that $j=i^w$. Because $j \in \Omega$ after the call we have $j = \mbox{\sc Rep}(j)$. Moreover,  from Lemma~\ref{repprop}, $j=\mbox{\sc Rep}(j) = \mbox{\sc Rep}(i)^w$ after the call. Because $\mbox{\sc Rep}(i) \le i$, this establishes Property~4.
\end{proof}

Combining the results in this section we have the following theorem.

\begin{thm} \label{Tprops} If {\sc Enumerate}$(S,R,M)$ completes, then $\mathcal E$ is complete and satisfies Properties~1--5.
\end{thm}

Note that even if {\sc Enumerate}$(S,R,M)$ returns a run limit exceeded statement, then the table produced up to that point still satisfies Properties 1--5. The table may even be complete, however, the secondary relations have not been scanned for all $i \in \Omega$.

\section{Complete tables and the rack isomorphism}

In this section we establish our main results regarding the relationship between $\mathcal E$ and $\mathcal R$ when {\sc Enumerate} completes. We begin with a useful lemma.

\begin{lem}
\label{omega} If {\sc Enumerate}$(S,R,M)$ completes, then for all $i,j \in \{1,2,\dots,\omega \}$ we have $\mbox{\sc Rep}(i)=\mbox{\sc Rep}(j)$ if and only if $\tau(i) = \tau(j)$ in $\mathcal R$. 
\end{lem}

\begin{proof} By Theorem~\ref{Tprops} we know that $\mathcal E$ satisfies Properties 1--5. Therefore, if $\mbox{\sc Rep}(i)=\mbox{\sc Rep}(j)$, then by Property 5 we have 
$$\tau(i) = \tau(\mbox{\sc Rep}(i))= \tau(\mbox{\sc Rep}(j))= \tau(j).$$
Conversely, assume that $\tau(i)=\tau(j)$.  Since $\mathcal E$ satisfies Properties~1--5 and {\sc Rep}$(i)$, {\sc Rep}$(j) \in \Omega$, there exists $a,b \in \Omega \cap \{1,2,\dots,g\}$ and $\alpha, \beta \in F(S)$ such that $a^\alpha = \mbox{\sc Rep}(i)$, $b^\beta = \mbox{\sc Rep}(j)$, 
$\tau(\mbox{\sc Rep}(i))=x_a^\alpha$, and $\tau(\mbox{\sc Rep}(j))=x_b^\beta$. We now have
$$x_a^{\alpha} = \tau(\mbox{\sc Rep}(i)) =\tau(i) =\tau(j) =\tau(\mbox{\sc Rep}(j))= x_b^{\beta}.$$ Therefore, there is a finite sequence of substitution moves that take $x_a^{\alpha}$ to $x_b^{\beta}$. We will show that for each substitution move, if $x_e^u$ is taken to $x_f^{v}$, then $\mbox{\sc Rep}(e)^u=\mbox{\sc Rep}(f)^{v}$ in $\mathcal E$.  Therefore, after the finite sequence of moves that takes $x_a^\alpha$ to 
$x_b^\beta$, we have 
$$\mbox{\sc Rep}(i) = a^\alpha =\mbox{\sc Rep}( a)^\alpha= \mbox{\sc Rep}( b)^\beta=b^\beta = \mbox{\sc Rep}(j).$$

\noindent Move (1): Assume move (1) takes $x_e^{uw}$ to $x_e^{uv \bar vw}$. Since $\mathcal E$ is complete and satisfies Property~2 we have that $\ell^{v \bar v}=\ell$ for all $\ell \in \Omega$. Therefore, 
$\mbox{\sc Rep}(e)^{uw}=\mbox{\sc Rep}(e)^{u v \bar v w}$.

\noindent Move (2): Assume move (2) takes $x_e^{uw}$ to $x_f^{w}$ where $x_e^u=x_f$ is a primary relation. Then {\sc Scan}$(\mbox{\sc Rep}(e),u,\mbox{\sc Rep}(f))$ was called in the {\sc Enumerate} procedure. Therefore, $\mbox{Rep}(e)^u=\mbox{\sc Rep}(f)$ in $\mathcal E$ and so $\mbox{Rep}(e)^{uw}=\mbox{Rep}(f)^{w}$.

\noindent Move (3): Assume move (3) takes $x_e^{v w}$ to $x_e^{v \bar u x_c u \bar x_d w}$ where $x_c^u=x_d$ is a primary relation. Then, {\sc Scan}$(\mbox{Rep}(e)^{v},\bar u x_c u \bar x_d, \mbox{Rep}(e)^{v})$ was called in the {\sc Enumerate} procedure because $\mbox{Rep}(e)^{v} \in \Omega$. Therefore, $\mbox{Rep}(e)^{v \bar u x_c u \bar x_d}=\mbox{Rep}(e)^{v}$ in $\mathcal E$ and so  $\mbox{Rep}(e)^{v \bar u x_c u \bar x_d w}=\mbox{Rep}(e)^{v w}$.
The other case of substitution move (3) is similar.
\end{proof}

Notice that if {\sc Enumerate}$(S,R,M)$ completes, then for all $j \in \Omega$ there exists $k \in \Omega \cap \{1,2,\dots,g\}$ and $w \in F(S)$ such that $j=k^w$.  In this case we may define two operations on $\Omega$ by
\begin{equation}\label{omegaops}
i^j = i^{\bar w x_k w} \ \ \ \mbox{and}\ \ \ i^{\bar j} = i^{\bar w \bar x_k w}.
\end{equation}

\begin{thm} \label{omegarack} If {\sc Enumerate}$(S,R,M)$ completes, then $\Omega$ with operations given by (\ref{omegaops}) is a rack and $\tau: \Omega \rightarrow \mathcal{R}$ is a rack isomorphism.
\end{thm}

\begin{proof} We first show the operations are well-defined.  Assume $j=k^w=\ell^v$ with $k, \ell  \in \Omega \cap \{1,2,\dots,g\}$ and $w,v \in F(S)$. Since $\mathcal E$ is complete, $k^w$ and $\ell^v$ are both in $\Omega$ and hence, by Lemma~\ref{omega},  $\tau(k^w)=\tau(\ell^v)$. Property 3 now implies that $\tau(k)^w=\tau(\ell)^v$ and by Property 1, $x_k^w=x_\ell^v$. Hence, the rack axioms tell us that $\tau(i)^{\bar w x_k w}=\tau(i)^{\bar v x_\ell v}$ and using Property 3 again, $\tau(i^{\bar w x_k w})=\tau(i^{\bar v x_\ell v})$. Now by Lemma~\ref{omega},  $i^{\bar w x_k w}=i^{\bar v x_\ell v}$.  Therefore $i^j$ is well-defined. The proof for $i^{\bar j}$ is similar.

Suppose that $i, j \in \Omega$ and $j=k^w$. We have
$$(i^j)^{\bar j}=(i^{\bar w x_k w})^{\bar j}=i^{\bar w x_k w\bar w\,\bar x_kw}=i.$$
Similarly, $(i^{\bar j})^j=i$. Thus, the first rack axiom holds.

Let  $\ell=m^v$. Then
$$(i^j)^\ell= (i^{\bar w x_k w})^\ell=i^{\bar w x_k w \bar v x_m v}$$
and
$$\left(i^\ell \right)^{\left(j^\ell \right)} = 
\left(i^{\bar v x_m v}\right)^{\left( j^{ \bar v x_m v}\right)}= \left(i^{\bar v x_m v}\right)^{\left( k^{w  \bar v x_m v}\right)}=i^{\bar v x_m v \bar v \bar x_m v \bar w x_k w \bar v x_m v}=
i^{\bar w x_k w \bar v x_m v}.$$
Therefore, the second rack axiom is satisfied.

The function $\tau:\Omega \to \mathcal R$ in injective by Lemma~\ref{omega}. Now suppose $x_i^w \in \mathcal R$ where $1 \le i \le g$. Then {\sc Rep}$(i) \in \Omega$ and  Properties~1--5 imply
$$\tau(\mbox{\sc Rep}(i)^w) = \tau(\mbox{\sc Rep}(i))^w=\tau(i)^w=x_i^w.$$
Therefore, $\tau$ is also surjective.

Finally, assume $i,j \in \Omega$ with $j=k^w$. Now $$\tau(i^j)=\tau(i^{\bar w x_k w})=\tau(i)^{\bar w x_k w}.$$  On the other hand $$\tau(i)^{\tau(j)}=\tau(i)^{\tau(k^w)}=\tau(i)^{(\tau(k)^w)}=\tau(i)^{(x_k^w)}=\tau(i)^{\bar w x_k w}.$$ Thus, $\tau(i^j)=\tau(i)^{\tau(j)}$. Hence,
$\tau(i) = \tau ( i^{\bar j j}) = \tau (i^{\bar j})^{\tau(j)}$. It now follows that $\tau(i^{\bar j})=\tau(i)^{\overline{\tau(j)}}$. Therefore, $\tau$ is a rack isomorphism.
\end{proof}

An important step in the {\sc Enumerate} procedure is the {\sc Define} command in line~\ref{ward} which represents Ward's modification of the Todd-Coxeter process in the rack setting.  This line requires that after scanning and filling all secondary relations for row $i \in \Omega$ we make additional definitions, if necessary, so that $i^y$ is defined for all $y \in S \cup \bar S$.  We do this before moving to the next live row. While this step can increases the size of $\Omega$, it has the benefit of producing a table that is filled in through row $i$ after completing all secondary relation scans for row $i$. This is important in the proof of the following theorem.

\begin{thm}\label{finite} If $\mathcal{R}=\langle S \mid R \rangle$ is a finite rack, then  {\sc Enumerate}$(S,R,M)$ completes for some $M$.
\end{thm}

\begin{proof}  Towards contradiction, assume $\mathcal R$ is finite and that {\sc Enumerate}$(S,R,M)$ returns a run limit exceeded statement for all $M \ge 1$. For a fixed $M$, let $\Omega_M$ be the live elements at the completion of
{\sc Enumerate}$(S,R,M)$ and define $\overline{\Omega} = \cap_{M \ge 1} \Omega_M$. Thus, $\overline \Omega$ is the set of elements that are not killed in any call of {\sc Coincidence}. By Property 1, we have $1 \in \overline{\Omega},$ so this set is nonempty. Now if $i \in \overline{\Omega}$ and $y \in S \cup \bar{S}$, then for all $M \ge i$ we have that line~\ref{ward} of {\sc Enumerate}$(S,R,M)$ guarantees that $i^y$ is defined. Notice that, as we increase $M$, the values of $i^y$ are nonincreasing since {\sc Coincidence} replaces dead values with their least representative. The values of  $i^y$ are  bounded below by $1$, therefore, at some point $i^y$ becomes stable. Since $S \cup \bar{S}$ is finite, this implies that given $i \in \overline{\Omega}$ there is an $M_i \ge i$ such that $i^y$ is defined and stable for all $M \ge M_i$ and for all $y \in S \cup \bar S$. Notice also that the stable value of $i^y$ is in $\overline{\Omega}$ for all $y \in S \cup \bar S$.

If $\overline \Omega$ were finite, then the set $\{M_i \mid i \in \overline{\Omega} \}$ would also be finite. In this case, we could choose $N \ge\max \{M_i \mid i \in \overline{\Omega} \}$ and {\sc Enumerate}$(S,R,N)$ would create a enumeration table in which $i^y$ is defined and $i^y \in \overline{\Omega}$ for all $i \in \overline{\Omega}$ and $y \in S \cup \bar S$. Hence, for all $w \in F(S)$, we would have $i^w \in \overline{\Omega}$ as well. Increase $N$, if necessary, so that $\Omega_{N}\cap \{1,2,\dots,g\}= \overline{\Omega}  \cap \{1,2,\dots,g\}$.  Now by Property~4, for any $n \in \Omega_{N}$,  there exists an $i \in \overline{\Omega} \cap \{1,2,\dots, g\}$ and $w \in F(S)$ such that $i^w=n$. However, $i \in \overline \Omega$ implies $i^w = n \in \overline{\Omega}$ and, hence, $\overline{\Omega} = \Omega_{N}$. Since $N \ge M_i \ge i$ for all $i \in \overline \Omega = \Omega_N$, this implies that {\sc Enumerate}$(S,R,N)$ completes. Hence we have a contradiction. Therefore, $\overline \Omega$ must be infinite. 

Now consider the infinite enumeration table $\mathcal T_{\infty}$ whose (infinitely many) rows are the elements of $\overline{\Omega}$ and whose entries are the stable values of $i^y$ for $i \in \overline{\Omega}$. This table is complete and satisfies Properties 1--5. Therefore, by the argument in Theorem~\ref{omegarack}, we have that $\tau : \overline{\Omega} \rightarrow \mathcal R$ is a rack isomorphism. This contradicts that $\mathcal R $ is finite. \end{proof}

The following is an immediate corollary of Theorems~\ref{omegarack} and \ref{finite}.

\begin{cor} A finitely presented rack $\mathcal R = \langle S \mid R \rangle$ is finite if and only if there is an $M$ such that $\mbox{\sc Enumerate}(S,R,M)$ completes.
\end{cor}

\section{Modifications to {\sc Enumerate}}

Recall that the {\sc Define} command in line~\ref{ward} of the {\sc Enumerate} procedure represents Ward's modification to the Todd-Coxeter procedure in the rack setting. Ward's modification was motivated by examples, given in \cite{WA}, where the original Todd-Coxeter process failed to complete even though the subgroup index was finite. A similar example exists in the rack setting. Consider the rack $\mathcal R$ with presentation
\begin{equation}\langle a ,\ b\ \mid a^a=a,\ b^b=b,\ a^{baba\bar b}=a,\ b^{baba \bar b}=b\rangle.
\label{counterexample}
\end{equation}
Then {\sc Enumerate}$(S, R, 11)$ completes with $\mathcal E$ shown below
\medskip
\begin{center}
\begin{small}
\begin{tabular}{c|cccccc} 
$\mathcal E$ & $a$ & $b$ & $\bar{a}$ & $\bar{b}$ & $\tau$ & $p$\\ \hline
1 & 1 & 1 & 1 & 1 & $a$ & 1\\ 
2 & 6 & 2 & 6 & 2 & $b$ & 2 \\
6 & 2 & 6 & 2 & 6 & $b^a$ & 6
\end{tabular}
\end{small}
\end{center}
\medskip
From the table it is clear that $\mathcal R$ is an involutory quandle of order 3. However, if we omit the command in line~\ref{ward} of {\sc Enumerate}, then the process never completes. 

\begin{prop}\label{wardexample} If {\sc W}$(S,R,M)$ is the rack enumeration procedure with line~\ref{ward} omitted,  then there is no $M$ for which {\sc W}$(S,R,M)$ completes for the finite rack presented by (\ref{counterexample}).
\end{prop}

Strong evidence for the veracity of the proposition can be obtained by coding the {\sc Enumerate} procedure with line~\ref{ward} omitted and running it for presentation (\ref{counterexample}) with large values of $M$. A formal proof of the remark can be given by using induction to prove that there is a sequence $4=n_0<n_1<n_2<n_3< \cdots$ with the property that, for all $i \ge 1$,  {\sc W}$(S,R,n_i-1)$ produces a run limit exceeded statement and an incomplete table which contains the following lines (here, a $0$ in $\mathcal E$ represents an undefined entry).
$$\begin{small}
\begin{array}{c|cccc|c} 
\mathcal T & a & b & \bar a& \bar b & \rho\\ \hline
\multicolumn{6}{c}{\vdots} \\
n_{i-1} & n_i & * & * & 0 & 4 \\
\multicolumn{6}{c}{\vdots} \\
n_i & 0 & n_i+1 & n_{i-1} & 0 & n_i \\
n_i+1& n_i+2 & 0 & n_i+3 & n_i & n_i+1 \\
n_i+2 & 0 & n_i+3 & n_i+1 & 0 & n_i+2 \\
n_i+3 & n_i+1 & 0 & 0 & n_i+2 & n_i+3 \\
\multicolumn{6}{c}{\vdots}
\end{array}
\end{small}$$
The inductive step requires a careful analysis of the helper tables for the scans of the secondary relations for rows $1$ through $n_{i+1}-1$ and how they affect the entries of $\mathcal E$. We leave the details for the interested reader.

The inclusion of line~\ref{ward} is a simple way to avoid the problem in Proposition~\ref{wardexample}, however, there may be other ways to alter {\sc Enumerate} to achieve this.  For example, our {\sc Init} routine does not attempt to produce the most efficient set of secondary relations. Notice that if $x^w=x$ for all $x$, then $x^{\bar w} =x$ for all $x$ as well.  For this reason it is unnecessary to have both $w$ and $\bar w$ in the set $R_2$ of secondary relations.  Similarly, if $x^{u v}=x$ is a secondary relation and given any $y \in \mathcal R$, if we let $x = y^{\bar u}$, then the secondary relation gives $y^{\bar u u v} = y^u$. Hence, $y^{vu}=y$ for all $y \in \mathcal R$.  Therefore, given any secondary relation $x^w=x$, we are free to cyclically permute the letters in $w$ to obtain an equivalent secondary relation.  This allows us to record any secondary relation $x^w=x$ with the unique word $w'$ which is {\em minimal} amongst all words obtained from $w$ and $\bar w$  by cyclic permutation and reduction. Here minimal means of shortest length and, among words of the same length, lexicographically smallest where the order on $S \cup \bar S$ is $x_1 < x_2 < ... < x_g < \bar x_g  < ... < \bar x_1$. 

Consider once more the rack $\mathcal R$ defined by (\ref{counterexample}).  With {\sc Init} defined by Algorithm~\ref{init}, the set of secondary relation words is $R_2 = \{b \bar a \bar b \bar a \bar b a b a b a  \bar b \bar a, b \bar a \bar b \bar a b a b a \bar b \bar b\}$. On the other hand, if instead we consider the minimal representatives of these words, then our secondary relations would be 
$R'_2 = \{a b a b a  \bar b \bar a b \bar a \bar b \bar a \bar b , a b a b \bar a \bar b \bar a \bar b\}$. Running a modification of {\sc Enumerate} with line~\ref{ward} omitted and $R_2$ replaced by $R'_2$, the process completes. We do not know if this is true in general, that is, if the secondary relations are chosen in this way, then is line~\ref{ward} still necessary?

As mentioned in the introduction, the Todd-Coxeter process was designed to find the index of a finitely generated subgroup of a group. That is, it  is designed to enumerate cosets. The order and Cayley graph for the group can be found by enumerating the cosets of the trivial subgroup. It is natural to ask if the \mbox{\sc Enumerate} process can be modified to enumerate something more general than the elements of the rack.  The natural analogy is to consider a finitely generated subrack $\Sigma \subseteq \mathcal R = \langle S \mid R \rangle$. If $x \in \mathcal R$, then define the \mbox{\bf rack coset} $\Sigma^x$ to be the set of elements $\{\sigma^x \mid \sigma \in \Sigma\}$. Unfortunately, the collection of all rack cosets of a given subrack does not, in general, partition the rack. We consider three interesting examples.

First, consider the fundamental 4-quandle of the righthand trefoil knot given by the presentation
$$\mathcal R_1 = \langle a, b \mid a^a=a, b^b=b, a^{bbbb}=a, b^{aaaa}=b, a^{ba}=b, b^{ab}=a \rangle.$$
From the {\sc Enumerate} process we find that $\mathcal R_1 = \{a,b,a^b,b^a,a^{bb},b^{aa} \}$. Moreover, it is not hard to show that $\Sigma = \langle a, a^{bb} \rangle = \{a, a^{bb}\}$ is a subrack (in fact, a subquandle). By direct calculation, this subrack has three disinct cosets
$\Sigma$, $\Sigma^b$, and $\Sigma^{ba}$ which partition $\mathcal R_1$. The {\sc Enumerate} process can be modified to enumerate these cosets and determine the action of $\mathcal R_1$ on the cosets.  Namely, initialize the process with $1$ representing the coset $\Sigma$.  Since $\Sigma$ is a subrack, it is fixed by the action of both generators $a$ and $a^{bb}$. So the modified process first scans these two subrack generator relations: $1^a=1$ and $1^{\bar b \bar b a b b}=1$. Next the modified process scans all secondary relations (from the presentation of $\mathcal R_1$) and an additional secondary relation $i^{\bar b \bar b a b b \bar a}=i$, since the action by different generators of $\Sigma$ should be the same, for all live $i$.  The modified process completes and enumerates the three distinct cosets given above.

As a second example, consider the involutory quandle of the $(2,4)$-torus link which has a presentation
$$\mathcal R_2 = \langle a, b \mid a^a=a, b^b=b, a^{bb}=a, b^{aa}=b, a^{bab}=a, b^{aba}=b \rangle.$$
The {\sc Enumerate} process determines that $\mathcal R_2 = \{a,b,a^b,b^a \}$. This rack contains the subrack $\Sigma = \langle a, a^b \rangle = \{a, a^b\}$. Notice that $\Sigma = \Sigma^a = \Sigma^b = \Sigma^{a^b}= \Sigma^{b^a}$ and so the cosets of the subrack do not partition $\mathcal R_2$.  On the other hand, consider the rack with presentation 
$$\mathcal R_3 = \langle a, b, c \mid a^a=a, b^b=b, c^c=c, a^b=a, a^c=a, b^a=b, ,c^a=c, b^{ccb}=c, b^{\bar cbb}=c \rangle$$
which has order six. The subrack $\Sigma = \langle a , b\rangle = \{a, b\}$ has five distinct cosets all of which contain $a$ and whose union is the entire rack. It is not immediately clear how the enumeration process can be modified in these last two examples in order to enumerate the distinct cosets.

\section{An application to knot theory}

We close with a sample calculation related to knot theory. Associated to every link is its fundamental quandle, which, of course, is a rack. However, the quandle of a knot or link is almost always infinite. If we pass to the quotient  2-quandle,  then there are many knots and links for which this is finite. A complete list of links with finite $n$-quandles for some $n$ is given in \cite{HS}. One such link is shown in Figure~\ref{link}. 

\begin{figure}[htbp]
\vspace*{13pt}
\centerline{\includegraphics*[scale=.6]{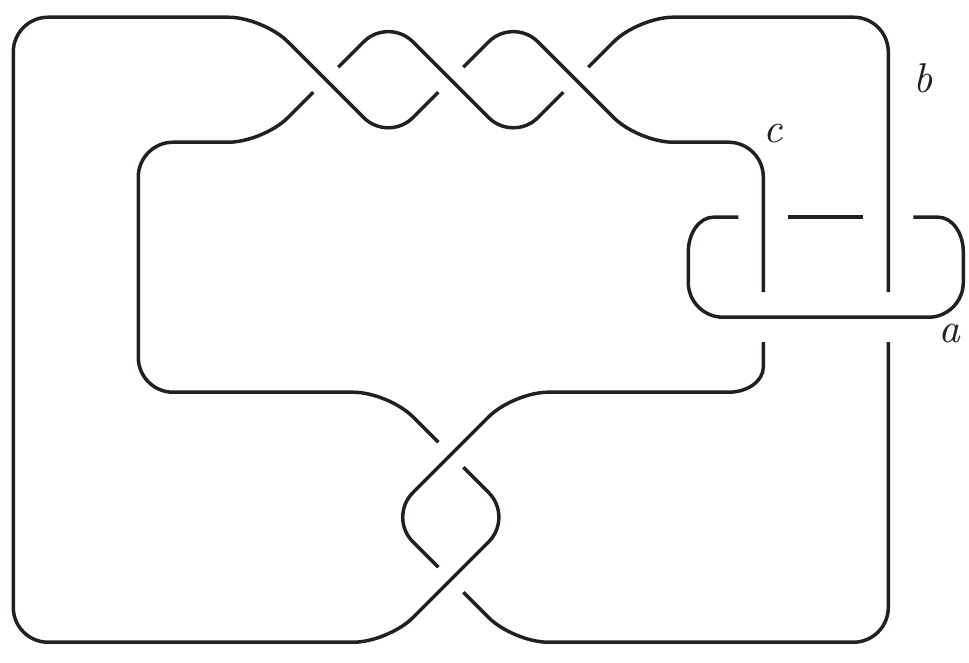}}
\caption{A link with finite 2-quandle.}
\label{link}
\end{figure}
A presentation for the 2-quandle of the link can be obtained from the diagram by labeling each arc of the diagram with a generator and then recording one relation at every crossing as indicated in Figure~\ref{relation}. In addition to these relations, we must also include the relations $x^x=x$ for every generator  and $x^{yy}=x$ for every pair of distinct generators $x$ and $y$. See \cite{HS} for more information on presentations of 2-quandles of links. If we use one generator for each arc, we will create a presentation with redundant generators. Instead, it is always possible to label some subset of the arcs with generators and then use the relations at each crossing to derive the labels on all of the other arcs. Arcs for which labels can be so derived in two different ways then give rise to the necessary relations.

\begin{figure}[htbp]
\vspace*{13pt}
\centerline{\includegraphics*[scale=.45]{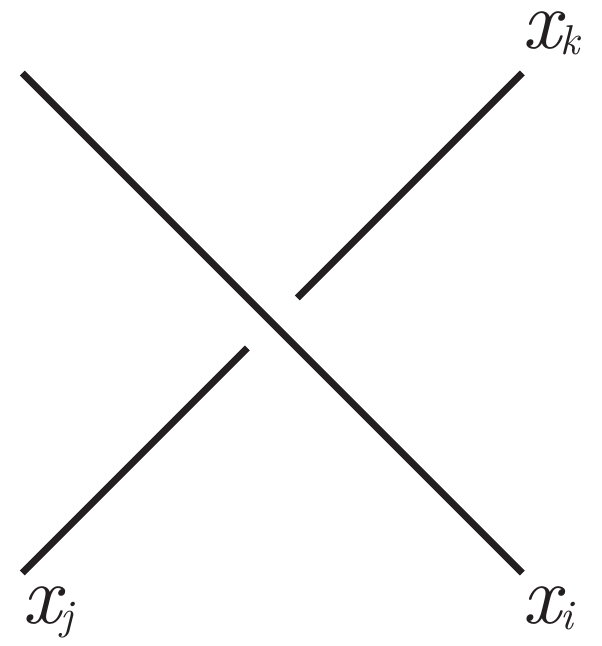}}
\caption{This crossing gives the relation $x_j^{x_i}=x_k$.}
\label{relation}
\end{figure}
If we label the three arcs shown in Figure~\ref{link} with the generators $a$, $b$, $c$, and follow the above procedure, we obtain the presentation
\begin{equation*}
\begin{split}\mathcal R=\langle a, b, c, &\mid  a^a=a, b^b=b, c^c=c,\\ 
&\quad a^{bb}=a, a^{cc}=a, b^{aa}=b, b^{cc}=b, c^{aa}=c, c^{bb}=c,\\ &\quad a^{bc}=a, b^{acbcbca}=c, b^{cbcacacb}=c \rangle.
\end{split}
\end{equation*}

Applying {\sc Enumerate} to this presentation yields a finite 2-quandle with two algebraic components corresponding to the  two components of the link. One algebraic component has four elements including the generator $a$ and the other has twenty elements including the generators 
$b$ and $c$. The Cayley graph of the 2-quandle can be immediately derived from the enumeration table and is shown in Figure~\ref{cayley}. The generators $a, b$, and  $c$, correspond to the solid, dashed, and dotted edges, respectively.

\begin{figure}[htbp]
\vspace*{13pt}
\centerline{\includegraphics*[scale=.85]{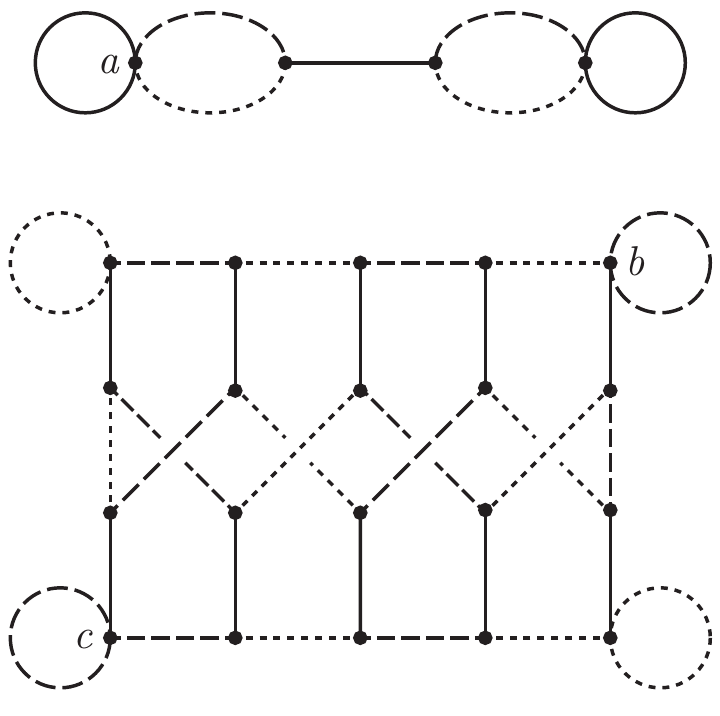}}
\caption{The 2-quandle of the link in Figure~\ref{link}.}
\label{cayley}
\end{figure}

\section{Implementation}

Since its implementation on a computer, there have been multiple modifications made to the Todd-Coxeter process that decrease the run-time or memory usage. In \cite{HO}, Holt characterizes the performance of a coset enumeration process in terms of the maximum number of live elements at any stage of the process. That is, the maximum value of $|\Omega|$ at any point. He also remarks that the total number of cosets defined would also be a reasonable measure. Holt declares a procedure to perform well if $\max \{|\Omega| \}$ is roughly less than $125\%$ of the index $[G:H]$.

We apply this analysis to the enumeration of involutory quandles of a family of links. In \cite{HS2}, it was shown that the order of the involutory quandle of the $(1/2,1/2,p/q;e)$-Montesinos link is $2(q+1)|(e-1)q-p|$. A selection of 21, 15, and 16, such quandles with orders near 10000, 20000, and 30000, respectively, were used. Run-times varied from 3.8 to 16.3 seconds for the first group, 18.7 to 41.3 seconds for the second group, and 24.2 to 166.4 seconds for the last group. Not surprisingly, run-times were roughly proportional to the number of elements defined during execution. This number ranged from about 40 to 320 times the order of the quandle. The largest number of live elements during execution was generally between  12 to 145 times the order of the quandle.  We coded our implementation of {\sc Enumerate} using {\em Python} and our program may be downloaded from the CompuTop.org software archive. A small selection of data is included in Table~\ref{performance data}. In this table, $t$ is the run-time in seconds, $L$ is the maximum value  of $|\Omega|$ at any point in the process, $E$ is the total number of quandle elements defined by the process, and $\mathcal O$ is the order of the quandle. We suspect that our procedure could be improved in order to perform well with respect to Holt's measure.

\begin{table}
$$
\begin{array}{rrrrrrrrr}
 p&q&e&t&\mathcal O&L&E&L/ \mathcal O&
E/ \mathcal O\\
\hline
 2 & 23 & 2 & 0.11 & 1008 & 4646 & 8109 & 4.6 & 8.0 \\
 53 & 61 & 2 & 0.20 & 992 & 9832 & 22148 & 9.9 & 22.3 \\
2 & 49 & -1 & 4.61 & 10000 & 117876 & 615021 & 11.8 & 61.5 \\
 2 & 11 & 5 & 0.24 & 1008 & 13893 & 30482 & 13.8 & 30.2 \\
2 & 61 & 5 & 25.84 & 30008 & 559137 & 2483138 & 18.6 & 82.7 \\
 4 & 41 & 4 & 4.91 & 9996 & 198559 & 593150 & 19.9 & 59.3 \\
 31 & 39 & 5 & 7.56 & 10000 & 367832 & 1039894 & 36.8 & 104.0 \\
 4 & 49 & -3 & 18.73 & 20000 & 838911 & 2312936 & 41.9 & 115.6 \\
 5 & 9 & -4 & 0.46 & 1000 & 47906 & 64245 & 47.9 & 64.2 \\
 19 & 45 & -1 & 10.11 & 10028 & 500924 & 1132344 & 50.0 & 112.9 \\
 27 & 53 & 5 & 28.64 & 19980 & 1194349 & 3942721 & 59.8 & 197.3 \\
39 & 64 & -2 & 87.98 & 30030 & 1954031 & 4726305 & 65.1 & 157.4 \\
 19 & 52 & 5 & 34.04 & 20034 & 1394756 & 4635357 & 69.6 & 231.4 \\
  25 & 64 & 5 & 128.65 & 30030 & 2321654 & 8237209 & 77.3 & 274.3 \\
 31 & 57 & -3 & 126.44 & 30044 & 3027595 & 7312811 & 100.8 & 243.4 \\
 12 & 43 & -4 & 35.56 & 19976 & 2069917 & 4988150 & 103.6 & 249.7 \\
  16 & 39 & -5 & 41.27 & 20000 & 2507287 & 5651463 & 125.4 & 282.6 \\
17 & 27 & -5 & 16.27 & 10024 & 1334984 & 2521252 & 133.2 & 251.5 \\
 31 & 47 & -5 & 166.36 & 30048 & 4338376 & 9511360 & 144.4 & 316.5 \\
\end{array}
$$
\caption{Performance of {\sc Enumerate} for a sample of Montesinos link quandles.}
\label{performance data}
\end{table}

\end{document}